\documentclass{amsart}
\usepackage[english]{babel}
\usepackage{amssymb,enumerate}

\newcommand{\mathd}{\mathrm{d}}
\newcommand{\mathe}{\mathrm{e}}
\newcommand{\tmop}[1]{{\operatorname{#1}}}
\newenvironment{enumerateroman}{\begin{enumerate}[\textup{(}i\textup{)}] }{\end{enumerate}}
\newtheorem{lemma}{Lemma}
\newtheorem{proposition}{Proposition}

\newtheorem{theorem}{Theorem}
\newenvironment{proof*}[1]{\noindent\textit{#1.\ }}{\hspace*{\fill}$\Box$\medskip}

\newcommand{\dt}{\frac{\partial}{\partial t}}
\newcommand{\normho}{| \mathring{h} |}
\newcommand{\Psum}{\underset{i, j}{\underset{\alpha > 1}{\sum}} \left( \mathring{h}^{\alpha}_{i j} \right)^2}
\newcommand{\limmax}{\underset{t \rightarrow - \infty}{\overline{\lim}}\! \max\limits_{M_t}}

\begin{document}

\title{Ancient Solution of Mean Curvature Flow in Space Forms}

\author{Li Lei}

\address{Center of Mathematical Sciences, Zhejiang University, Zhejiang Hangzhou,
310027, People's Republic of China}

\email{lei-li@zju.edu.cn}

\author{Hongwei Xu}

\address{Center of Mathematical Sciences, Zhejiang University, Zhejiang Hangzhou,
310027, People's Republic of China}

\email{xuhw@zju.edu.cn}

\author{Entao Zhao}

\address{Center of Mathematical Sciences, Zhejiang University, Zhejiang Hangzhou,
310027, People's Republic of China}
\email{zhaoet@zju.edu.cn}

\thanks{Research supported by the National Natural Science Foundation of China, Grant No.
    11531012; and the China Postdoctoral Science Foundation, Grant No. BX20180274.}



\begin{abstract}
In this paper we investigate the rigidity of ancient solutions of the mean curvature flow with arbitrary codimension in space forms. We first prove that under certain sharp asymptotic pointwise curvature pinching condition the ancient solution in a sphere is either a shrinking spherical cap or a totally geodesic sphere. Then we show that under certain pointwise curvature pinching condition the ancient solution in a hyperbolic space is a family of shrinking spheres. We also obtain a rigidity result for ancient solutions in a nonnegatively curved space form under an  asymptotic integral curvature pinching condition.

\end{abstract}

{\maketitle}

\section{Introduction}

Let $F:M^{n}\times(T_{1},T_{2})\rightarrow N^{n+p}$ be a smooth family
of immersions into a Riemannian manifold that satisfies
\begin{equation}
\frac{\partial}{\partial t}F(x,t)=H(x,t),\label{eq: MCF}
\end{equation}
where $H(x,t)$ is the mean curvature vector of the submanifold $M_{t}=F(M,t)$.
(\ref{eq: MCF}) is the negative gradient flow of the volume functional
on the submanifolds and $F(x,t)$ is called the solution of the mean
curvature flow.

The mean curvature flow of hypersurfaces has been investigated extensively, see for instance the fundamental papers of Huisken \cite{MR772132,Huisken-86-Invent Math,Huisken-87-MathZ} on the smooth convergence theory of the mean curvature flow. For the mean curvature flow of submanifolds with higher codimensions, fruitful results were obtained. For example, following the work of Huisken \cite{MR772132,Huisken-86-Invent Math,Huisken-87-MathZ}, Andrews and Baker \cite{AB-JDG,baker2011mean}, Liu, Xu, Ye and Zhao \cite{LXYZ,LXZ} proved several smooth convergence theorems for the mean curvature flow of submanifolds of arbitrary codimensions under pointwise curvature pinching conditions.  Motivated by the rigidity theorem for closed submanifolds with parallel mean curvature vector and the topological sphere theorem for complete submanifolds in space forms \cite{SX,Xu-PhD,Xu-Arch Math}, it was proposed in \cite{LXYZ-2,LXZ-ICCM} that the mean curvature flow of closed submanifolds satisfying  initial curvature condition  $|h|^{2}\leqslant\alpha(n,|H|,c)$ and $|H|^{2}+n^{2}c>0$  in a complete simply connected  space form $\mathbb{F}^{n+p}(c)$ with constant curvature $c\neq 0$ would converges smoothly to a round point in finite time or to a totally geodesic sphere in $\mathbb{F}^{n+p}(c)$ with $c>0$.  Here $h$ denotes the second fundamental form of the submanifold and $\alpha(n,|H|,c)=nc+\frac{n|H|^{2}}{2(n-1)}-\frac{n-2}{2(n-1)}\sqrt{|H|^{4}+4(n-1)c|H|^{2}}$.  It is easy to check that $\alpha(n,|H|,c)$ is strictly bigger than $\frac{1}{n-1}|H|^{2}+2c$.   Lei-Xu \cite{Lei-Xu-1} proved the smooth convergence theorem of mean curvature flow in $\mathbb{F}^{n+p}(c)$ with $c<0$ and $n\geqslant 6$ under the curvature condition   $|h|^{2}\leqslant\alpha(n,|H|,c)$ and $|H|^{2}+n^{2}c>0$. For the case $c>0$,  Lei-Xu \cite{Lei-Xu-3} proved the smooth convergence theorem of mean curvature flow in  $\mathbb{F}^{n+p}(c)$  ($n\geqslant 6$) under a sharp curvature condition   $|h|^{2}\leqslant\gamma(n,|H|,c)$.
Meanwhile, some smooth convergence theorems for the mean curvature flow of arbitrary codimension in space forms under integral curvature pinching conditions were proved in \cite{LXYZ-2,LXZ-1}. See \cite{Lei-Xu-3-1,LXZ-ICCM} for recent progress in the smooth convergence theory of mean curvature flow of arbitrary codimensions.

In the present paper, we focus on the ancient solution of the mean curvature flow, which is the solution of (\ref{eq: MCF}) on the time interval $(-\infty,T)$ for some $T<\infty$. In recent years, various researchers have investigated ancient solutions of the mean curvature flow of codimension one. Closed convex ancient solutions of the curve shortening flow in the plane have been completely classified by Daskalopoulos-Hamilton-Sesum \cite{DHS-JDG} to be either shrinking round circles or Angenent ovals. In higher dimension, a rigidity theorem was proved in \cite{HS2015} stating  that a closed and convex ancient solution of the mean curvature flow on $(-\infty,0)$ in the Euclidean space $\mathbb{R}^{n+1}$ with principal curvatures $\lambda_i$ and mean curvature function $|H|$ of $M_{t}$ satisfying $\lambda_{i}\geqslant\epsilon |H|$ for a positive constant $\epsilon$ is a family of shrinking spheres. For the ancient solution in the sphere, Bryan-Louie \cite{BL} have proved that the only closed, embedded and convex ancient solutions to the curve shortening flow in the unit sphere $\mathbb{S}^{2}$ are equators or shrinking circles,  and for higher dimensional case some characterizations of spherical ancient solution  in terms of curvature pinching have been given in \cite{HS2015}.  In higher codimensional case, it was proved in \cite{LN2017,RS2019} that the ancient solution to the mean curvature flow in the Euclidean space $\mathbb{R}^{n+p}$ satisfying certain pointwise curvature pinching condition uniformly on $(-\infty,0)$ is a family of shrinking spheres, and similar rigidity phenomenon holds for the ancient solutions in the unit sphere $\mathbb{S}^{n+p}$ under the pinching condition $|h|^2\leqslant \frac{1}{n-1}|H|^2+2$ for $n\geqslant 4$ and similar but stronger pinching condition for $n=2,3$.  It should be remarked that the curvature pinching conditions in \cite{LN2017,RS2019} are proposed according to the convergence theorems for the mean curvature flow in \cite{AB-JDG,baker2011mean}.  On the other hand,  Brendle-Choi \cite{BC} proved that the rotationally symmetric bowl soliton is the only complete noncompact ancient solution of mean curvature flow in $\mathbb{R}^{3}$ which is strictly convex and noncollapsed, and they \cite{BC-1} also generalized this result to higher dimensional case assuming additionally that the solution is uniformly two-convex. For more results about ancient solutions of the mean curvature flow and other geometric flows, see \cite{ADS-JDG,BUS-DukeMJ,BIS,CM,Daskalopoulos-PCM,Haslhofer-Hershkovits-CAG,XJWang-AnnMath}, et.al.

In the late 1960's, Simons \cite{Simons}, Lawson \cite{Lawson}, and Chern-do Carmo-Kobayashi \cite{CdoK} proved a famous rigidity theorem for $n$-dimensional compact minimal submanifolds in the unit sphere $\mathbb{S}^{n+p}$ under the pinching condition $|h|^2 \leqslant n/\big(2-\frac{1}{p}\big)$. In \cite{LL}, Li-Li improved the pinching constant to $n/\big(1 + \frac{1}{2} \tmop{sgn} (p - 1)\big)$. Inspired by the Simons-Lawson-Chern-do Carmo-Kobayashi-Li-Li rigidity theorem, we prove the following theorem for ancient solutions of  mean curvature flow in the sphere.
\begin{theorem} \label{thm-1}
  Let $F : M^n \times (- \infty, 0) \rightarrow \mathbb{S}^{n + p}$ be a
  compact ancient solution of mean curvature flow in the unit sphere. If
  \[ \limmax | h |^2 <
     \frac{n}{1 + \frac{1}{2} \tmop{sgn} (p - 1)}, \]
  then $\limmax | h |^2=0$
  and $M_t$ is either a shrinking spherical cap or a totally geodesic sphere.
\end{theorem}

The term ``spherical cap" means a family of shrinking totally umbilical spheres.
Our pinching condition in Theorem \ref{thm-1} is sharp. In fact,  the Clifford minimal hypersurfaces in  $\mathbb{S}^{n + 1}$ satisfy $|h|^2=n$, and  the Veronese minimal surface in $\mathbb{S}^4$ satisfies $|h|^2=\frac{4}{3}$. The Clifford hypersurfaces and Veronese surface are obviously static ancient solutions of the mean curvature flow, so the pinching condition for the case $p=1$ or the case $n=2$ and $p=2$ in Theorem \ref{thm-1} is optimal.

Theorem \ref{thm-1}  is a direct consequence of the following
\begin{theorem} \label{thm-2}
  Let $F : M^n \times (- \infty, 0) \rightarrow \mathbb{S}^{n + p}$ be a
  compact ancient solution of mean curvature flow in the unit sphere. 
  Suppose $F$ satisfies one of the following conditions:
  \begin{enumerateroman}
    \item $\limmax \left( | h |^2 - \min \left\{ \tfrac{3}{n + 2}, \tfrac{4 (n -
    1)}{n (n + 2)} \right\} | H |^2 \right) < n$, for $p = 1$,
    
    \item $\limmax \left( | h |^2 - \frac{4}{3 n} | H |^2 \right) < \frac{2 n}{3}$, for $p = 2$,
    
    \item $\limmax \left( | h |^2 - \frac{4}{3 n} | H |^2
    \right) < \frac{3 n}{5}$ or    
    $\limmax \left( | h |^2 - \frac{1}{n} | H |^2
    \right) < \frac{2 n}{3}$, for $p \geqslant 3$.
  \end{enumerateroman}
  Then $M_t$ is either a shrinking spherical cap or a totally geodesic sphere.
\end{theorem}

The curvature condition in Theorem \ref{thm-2} is obviously weaker than those  in  \cite{LN2017,RS2019}.
It should be mentioned that the pinching constant in Theorem \ref{thm-2} is sharp.

If the ambient space is the hyperbolic space   $\mathbb{H}^{n+p}$ with constant curvature $-1$, we prove the following rigidity theorem.

\begin{theorem} \label{thm-hyperbolic}
  Let $F : M^n \times (- \infty, 0) \rightarrow \mathbb{H}^{n + p}$ be a
  compact ancient solution of mean curvature flow in the hyperbolic space.
  Suppose there exists a positive number $\varepsilon$, such that for all $t <
  0$, $M_t$ satisfies $| H | > n$ and
  \[ | \mathring{h} |^2 \leqslant k | H |^2  \left( 1 - \tfrac{n^2}{| H |^2}
     \right)^{2 + \varepsilon}, \]
  where
  \[ k = \begin{cases}
       \frac{1}{3 n}, & p \geqslant 2 \enspace \text{and} \enspace n \geqslant
       7,\\
       \frac{n - 1}{2 n (n + 2)}, & \text{otherwise} .
     \end{cases} \]
  Then $M_t$ is a family of shrinking spheres.
\end{theorem}

In Theorem \ref{thm-hyperbolic}, $\mathring{h}$ is the tracefree second fundamental form of a submanifold.
As far as we know, Theorem \ref{thm-hyperbolic} is the first result in the literature about the rigidity of ancient solutions of mean curvature flow in the hyperbolic space. Now we give more details on the shrinking spheres in the hyperbolic space. A geodesic sphere of radius $r$ in the hyperbolic space has mean curvature $n \coth r$. Thus, the family of geodesic spheres whose radii are given by $r (t) = \tmop{arccosh}\, \mathe^{- n t}$ is an ancient solution of mean curvature flow. Furthermore, the mean curvature functions of the geodesic spheres are $|H| (t) = n \mathe^{- n t} / \sqrt{\mathe^{- 2 n t} - 1}$. We see that $|H| (t)$ are always larger than $n$, and tend to $n$ as $t \rightarrow - \infty$.

The rigidity of submanifolds under integral curvature pinching condition is also an attractive topic in submanifold theory. For instance, Shiomaha, Xu and Gu \cite{SX-2,Xu-JMSJ,Xu-Gu-CAG} proved several geometric and topological rigidity theorems for submanifolds in $\mathbb{F}^{n + p}(c)$ with $c\geqslant 0$  under the pinching condition of the curvature integral $ \int_{M} | \mathring{h} |^n \mathd \mu$. This curvature integral is called the Willmore functional.  It is invariant under conformal transformations of the ambient space and is a higher dimensional generalization of the classical Willmore functional for 2-dimensional surfaces. Inspired by these global rigidity theorems, we prove the following
\begin{theorem}\label{thm-integral}
  Let  $F : M^n \times (- \infty, 0) \rightarrow \mathbb{F}^{n +  p}(c)$ be a compact ancient solution of mean curvature flow in a space form $\mathbb{F}^{n+p}(c)$ with $c\geqslant 0$.
  Suppose there holds
   \[
    \underset{t \rightarrow - \infty}{\overline{\lim}} \int_{M_t} | \mathring{h} |^n \mathd \mu_t < C (n),
    \]
  where $C (n)$ is a positive constant explicitly depending on $n$. Then $M_t$
  is either a shrinking sphere or a totally geodesic sphere in $\mathbb{F}^{n +  p}(c)$ with $c>0$.
\end{theorem}

Theorem \ref{thm-integral} is also the first result as far as we know in the literature about the rigidity of ancient solutions of mean curvature flow under integral curvature pinching condition.

The paper is organized as follows. In Section 2, we recall the evolution equations
along the mean curvature flow in space forms and derive several curvature inequalities which
will be used in the proofs of theorems. In Section 3, we first prove several rigidity theorems for ancient solutions of mean curvature flow in the sphere, which combined together verify Theorems  \ref{thm-1} and \ref{thm-2}, then we give the proof of Theorem \ref{thm-hyperbolic}. Theorem \ref{thm-integral} will be proved in Section 4.

\section{Preliminaries}

Let $M^n$ be an $n$-dimensional Riemannian submanifold immersed in an $(n+p)$-dimensional simply connected space
form $\mathbb{F}^{n + p} (c)$ with constant curvature $c$. We denote by
$\nabla$ and $\overline{\nabla}$ the Levi-Civita connections of the
submanifold $M$ and the ambient space $\mathbb{F}^{n + p} (c)$, respectively.
The second fundamental form of $M$ is defined as
\[ h (u, v) = \overline{\nabla}_u v - \nabla_u v \]
for tangent vector fields $u, v$ on $M$.

Let $\{ e_i  | 1 \leqslant i \leqslant n \}$ be a local orthonormal frame for
the tangent bundle and $\{ \nu_{\alpha}  | 1 \leqslant \alpha \leqslant p \}$
be a local orthonormal frame for the normal bundle. In a local frame, we set
$h (e_i, e_j) = \sum_{\alpha} h^{\alpha}_{i j} \nu_{\alpha}$. The mean
curvature vector is given by
\[ H = \sum_{\alpha} H^{\alpha} \nu_{\alpha}, \quad H^{\alpha} = \sum_i
   h^{\alpha}_{i i} . \]
Let $\mathring{h} = h - \tfrac{1}{n} H \otimes g$ be the tracefree second
fundamental form. Its norm has the relation $\normho^2 = | h |^2 - \frac{1}{n}
| H |^2$.

Let $F : M^n \times I \rightarrow \mathbb{F}^{n + p} (c)$ be a mean curvature
flow. Let $M_t = F (M, t)$ be the submanifold at $t$. Andrews and Baker
derived the following evolution equations along the mean curvature flow
{\cite{AB-JDG,baker2011mean}}.

\begin{lemma}
  \label{evo}For the mean curvature flow $F : M^n \times I \rightarrow
  \mathbb{F}^{n + p} (c)$, we have
  \begin{enumerateroman}
    \item $\dt | h |^2 = \Delta | h |^2 - 2 | \nabla h |^2 + 2 R_1 + 4 c | H
    |^2 - 2 n c | h |^2$,

    \item $\dt | H |^2 = \Delta | H |^2 - 2 | \nabla H |^2 + 2 R_2 + 2 n c | H
    |^2$,

    \item $\dt \normho^2 = \Delta \normho^2 - 2 | \nabla \mathring{h} |^2 + 2 R_1 -
    \frac{2}{n} R_2 - 2 n c \normho^2$,
  \end{enumerateroman}
  where
  \[ R_1 = \sum_{\alpha, \beta} \Bigg( \sum_{i, j} h^{\alpha}_{i j}
     h^{\beta}_{i j} \Bigg)^2 + \sum_{i, j, \alpha, \beta} \Bigg( \sum_k
     (h^{\alpha}_{i k} h^{\beta}_{j k} - h^{\beta}_{i k} h^{\alpha}_{j k})
     \Bigg)^2, \]
  \[ R_2 = \sum_{i, j} \Bigg( \sum_{\alpha} H^{\alpha} h^{\alpha}_{i j}
     \Bigg)^2 . \]
\end{lemma}

In particular, if the codimension is 1, then $R_1 = | h |^4$, $R_2 = | H |^2 |
h |^2$.

Huisken {\cite{MR772132}} obtained the following estimate for the gradient
terms
\begin{equation}
  | \nabla h |^2 \geqslant \frac{3}{n + 2} | \nabla H |^2, \label{Delh2}
\end{equation}
which implies
\begin{equation}
  | \nabla \mathring{h} |^2 = | \nabla h |^2 - \frac{1}{n} | \nabla H |^2 \geqslant
  \frac{2 (n - 1)}{n (n + 2)} | \nabla H |^2 . \label{dh2}
\end{equation}

We will need the following matrix inequality due to Li-Li {\cite{LL}}.

\begin{proposition}
  \label{pmatrix}Let $A_1, \cdots, A_p$ be $p (\geqslant 2)$ symmetric
  matrices. Then
  \[ \sum_{\alpha, \beta} | A_{\alpha} A_{\beta} - A_{\beta} A_{\alpha} |^2 +
     \sum_{\alpha, \beta} [\tmop{tr} (A_{\alpha} A_{\beta})]^2 \leqslant
     \frac{3}{2} \Bigg( \sum_{\alpha} | A_{\alpha} |^2 \Bigg)^2 . \]
\end{proposition}

Using Proposition \ref{pmatrix}, we have
\begin{eqnarray}
  R_1 & = & \sum_{\alpha, \beta} \Bigg( \sum_{i, j} \mathring{h}^{\alpha}_{i j}
  \mathring{h}^{\beta}_{i j} \Bigg)^2 + \sum_{i, j, \alpha, \beta} \Bigg( \sum_k
  (\mathring{h}^{\alpha}_{i k}  \mathring{h}^{\beta}_{j k} - \mathring{h}^{\beta}_{i k}
  \mathring{h}^{\alpha}_{j k}) \Bigg)^2 \nonumber\\
  &  & + \frac{2}{n} \sum_{i, j} \Bigg( \sum_{\alpha} H^{\alpha}
  \mathring{h}^{\alpha}_{i j} \Bigg)^2 + \frac{1}{n^2} | H |^4 \nonumber\\
  & \leqslant & \frac{3}{2} | \mathring{h} |^4 + \frac{2}{n} \sum_{i, j} \Bigg(
  \sum_{\alpha} H^{\alpha} \mathring{h}^{\alpha}_{i j} \Bigg)^2 + \frac{1}{n^2} |
  H |^4  \label{R12R2/n}\\
  & = & \frac{3}{2} | \mathring{h} |^4 + \frac{2}{n} R_2 - \frac{1}{n^2} | H |^4 .
  \nonumber
\end{eqnarray}

At a point on $M$, we choose an orthonormal frame $\{ \nu_{\alpha} \}$ for the
normal space, such that $H = | H | \nu_1$, and an orthonormal frame $\{ e_i
\}$ for the tangent space, such that $h^1_{i j} = \lambda_i \delta_{i j}$.
Then $\left( \mathring{h}^1_{i j} \right)$ is also diagonal, we denote its
diagonal elements by $\mathring{\lambda}_i$. Thus $\mathring{\lambda}_i =
\lambda_i - \frac{1}{n} | H |$ and $\mathring{h}^{\alpha}_{i j} =
h^{\alpha}_{i j}$ for $\alpha > 1$.

With the special frame, $R_1$ becomes
\begin{eqnarray*}
  R_1 & = & \Bigg( \sum_i \lambda_i^2 \Bigg)^2 + 2 \sum_{\alpha > 1} \Bigg(
  \sum_i \mathring{\lambda}_i  \mathring{h}^{\alpha}_{i i} \Bigg)^2 +
  \sum_{\alpha, \beta > 1} \Bigg( \sum_{i, j} \mathring{h}^{\alpha}_{i j}
  \mathring{h}^{\beta}_{i j} \Bigg)^2\\
  &  & + 2 \underset{i \neq j}{\underset{\alpha > 1}{\sum}} \left( \left(
  \mathring{\lambda}_i - \mathring{\lambda}_j \right) \mathring{h}^{\alpha}_{i
  j} \right)^2 + \underset{i, j}{\underset{\alpha, \beta > 1}{\sum}} \Bigg(
  \sum_k \left( \mathring{h}^{\alpha}_{i k}  \mathring{h}^{\beta}_{j k} -
  \mathring{h}^{\alpha}_{j k}  \mathring{h}^{\beta}_{i k} \right) \Bigg)^2 .
\end{eqnarray*}
We use Cauchy-Schwarz inequality to get
\[ \sum_{\alpha > 1} \Bigg( \sum_i \mathring{\lambda}_i
   \mathring{h}^{\alpha}_{i i} \Bigg)^2 \leqslant \sum_{\alpha > 1} \Bigg(
   \sum_i \mathring{\lambda}_i^2 \Bigg) \Bigg( \sum_i \left(
   \mathring{h}^{\alpha}_{i i} \right)^2 \Bigg) . \]
We also have
\[ \underset{i \neq j}{\underset{\alpha > 1}{\sum}} \left( \left(
   \mathring{\lambda}_i - \mathring{\lambda}_j \right)
   \mathring{h}^{\alpha}_{i j} \right)^2 \leqslant \underset{i \neq
   j}{\underset{\alpha > 1}{\sum}} 2 \left( \mathring{\lambda}_i^2 +
   \mathring{\lambda}_j^2 \right) \left( \mathring{h}^{\alpha}_{i j} \right)^2
   \leqslant 2 \sum_{\alpha > 1} \Bigg( \sum_i \mathring{\lambda}_i^2 \Bigg)
   \Bigg( \sum_{i \neq j} \left( \mathring{h}^{\alpha}_{i j} \right)^2 \Bigg)
   . \]
Thus we get
\[ \sum_{\alpha > 1} \Bigg( \sum_i \mathring{\lambda}_i
   \mathring{h}^{\alpha}_{i i} \Bigg)^2 + \underset{i \neq
   j}{\underset{\alpha > 1}{\sum}} \left( \left( \mathring{\lambda}_i -
   \mathring{\lambda}_j \right) \mathring{h}^{\alpha}_{i j} \right)^2
   \leqslant 2 \Bigg( \sum_i \mathring{\lambda}_i^2 \Bigg) \Bigg( \Psum \Bigg) . \]
Using Proposition \ref{pmatrix}, we obtain
\[ \sum_{\alpha, \beta > 1} \Bigg( \sum_{i, j} \mathring{h}^{\alpha}_{i j}
   \mathring{h}^{\beta}_{i j} \Bigg)^2 + \underset{i, j}{\underset{\alpha,
   \beta > 1}{\sum}} \Bigg( \sum_k \left( \mathring{h}^{\alpha}_{i k}
   \mathring{h}^{\beta}_{j k} - \mathring{h}^{\alpha}_{j k}
   \mathring{h}^{\beta}_{i k} \right) \Bigg)^2 \leqslant \tilde{\xi}
   \Bigg( \Psum \Bigg)^2, \]
where $\tilde{\xi} = 1 + \frac{1}{2} \tmop{sgn} (p - 2)$.
Therefore, we obtain
\[ R_1 \leqslant \Bigg( \sum_i \lambda_i^2 \Bigg)^2 + 4 \Bigg( \sum_i
   \mathring{\lambda}_i^2 \Bigg) \Bigg( \Psum \Bigg) + \tilde{\xi} \Bigg( \Psum \Bigg)^2 . \]
We set
\[ P = \Psum . \]
This gives $\sum_i \lambda_i^2 = | h |^2 - P$ and $\sum_i
\mathring{\lambda}_i^2 = \normho^2 - P$. Then we obtain
\begin{equation}
  R_1 \leqslant | h |^4 + \left( 2 \normho^2 - \frac{2}{n} | H |^2 \right) P -
  \xi^{- 1} P^2,
  \label{R1nogreater}
\end{equation}
where $ \xi = \frac{2}{4 - \tmop{sgn} (p - 2)} $.
We also have
\begin{equation}
  R_2 = \sum_i (| H | \lambda_i)^2 = | H |^2 (| h |^2 - P) . \label{R2eq}
\end{equation}

\section{Pointwisely pinched ancient solutions in curved space forms}

In this section we consider the rigidity of ancient solutions of mean curvature flow in curved space forms under pointwise curvature pinching  conditions.

\subsection{Ancient solutions in spheres}\

We first consider codimension 1 ancient solutions and prove the following

\begin{theorem}\label{codim 1 sphere}
  Let $F : M^n \times (- \infty, 0) \rightarrow \mathbb{S}^{n + 1}$ be a
  compact ancient solution of mean curvature flow in the unit sphere. If
  \[ \sup_{M \times (- \infty, 0)} \left( | h |^2 - \min \left\{ \tfrac{3}{n +
     2}, \tfrac{4 (n - 1)}{n (n + 2)} \right\}  | H |^2 \right) < n, \]
  then $M_t$ is either a shrinking spherical cap or a totally geodesic sphere.
\end{theorem}

\begin{proof}
  Set
  \[ \gamma = \frac{2 (n - 1)}{n (n + 2)} \qquad\text{and}\qquad k = \min \left\{ \gamma, 2
     \gamma - \frac{1}{n} \right\} . \]
  From the pinching condition, there exists a positive constant $a$, such that
  for all $t<0$ there holds
  \[ | \mathring{h} |^2 - k | H |^2 < a < n. \]
  We study the function
  \[ f = \frac{\normho^2}{\gamma | H |^2 + a} . \]
  Notice that $f$ is always less than 1.

  We have
  \[ \dt f = \frac{1}{\gamma | H |^2 + a} \left( \dt \normho^2 - \gamma f \dt
     | H |^2 \right) \]
  and
  \[ \Delta f = \frac{1}{\gamma | H |^2 + a} \left( \Delta \normho^2 - \gamma
     f \Delta | H |^2 - 2 \gamma \langle \nabla f, \nabla | H |^2 \rangle
     \right) . \]
  Then we get from the evolution equations that
  \begin{eqnarray}
    &&\left( \dt - \Delta \right) f \nonumber\\& = & \frac{2 \gamma}{\gamma | H |^2 + a}
    \langle \nabla f, \nabla | H |^2 \rangle - \frac{2}{\gamma | H |^2 + a} (|
    \nabla \mathring{h} |^2 - \gamma f | \nabla H |^2) \nonumber\\
    &  & + 2 f \left[ | h |^2 - n - \frac{\gamma}{\gamma | H |^2 + a} (| H
    |^2 | h |^2 + n | H |^2) \right] .  \label{dtfinS1}
  \end{eqnarray}
  Firstly, we estimate the gradient terms in (\ref{dtfinS1}). From {\eqref{dh2}}
  we have
  \[ | \nabla \mathring{h} |^2 \geqslant \gamma | \nabla H |^2 \geqslant \gamma f |
     \nabla H |^2 . \]
  Secondly, we estimate the reaction terms.
  \begin{eqnarray*}
    &  & | h |^2 - n - \frac{\gamma}{\gamma | H |^2 + a} (| H |^2 | h |^2 + n |
    H |^2)\\
    & = & \frac{a | h |^2 - n (2 \gamma | H |^2 + a)}{\gamma | H |^2 + a}\\
    & \leqslant & \frac{a (2 \gamma | H |^2 + a) - n (2 \gamma | H |^2 +
    a)}{\gamma | H |^2 + a}\\
    & = & \frac{(a - n) (2 \gamma | H |^2 + a)}{\gamma | H |^2 + a}\\
    & \leqslant & a - n.
  \end{eqnarray*}
  Letting $\delta = n - a$, we obtain
  \[ \left( \dt - \Delta \right) f \leqslant \frac{2 \gamma}{\gamma | H |^2 +
     a} \langle \nabla f, \nabla | H |^2 \rangle - 2 \delta f. \]
  Applying the maximum principle, we get
  \[ \forall s < t, \quad \max_{M_t} f \leqslant \mathe^{-2 \delta (t - s)}
     \max_{M_s} f . \]
  Letting $s \rightarrow - \infty$, we obtain $\max_{M_t} f = 0$, i.e., $M_t$
  is a totally umbilical sphere.
\end{proof}

To derive the rigidity theorem in higher codimensions, we need the following inequality.

\begin{lemma}
  \label{Gxy}Let $b$ and $\xi$ be positive numbers. Define
  \[ G (x, y) = \frac{2 b \left( \frac{4}{3} x + b \right) + x (y - 1)}{x + 2
     b} - \frac{x y + \xi^{- 1} y^2}{\frac{1}{3} x + b} + 2 y - 1. \]
  If $\frac{1}{2} \leqslant \xi < b^{- 1} - 1$, then
  \[ \sup_{x, y \in [0, + \infty)} G (x, y) < 0. \]
\end{lemma}

\begin{proof}
  We write $G (x, y)$ as a quadratic function of $y$.
  \[ G (x, y) = \frac{2 b \left( \frac{4}{3} x + b \right) - x}{x + 2 b} - 1
     + \frac{b (7 x + 12 b)}{(x + 2 b)  (x + 3 b)} y - \frac{3 \xi^{- 1}}{x +
     3 b} y^2 . \]
  So, for $x \geqslant 0$, we have
  \begin{eqnarray}
    \sup_{y \in \mathbb{R}} G (x, y) & = & \frac{2 b \left( \frac{4}{3} x + b
    \right) - x}{x + 2 b} - 1 + \frac{\xi}{12 (x + 3 b)} \left[ \frac{b (7 x +
    12 b)}{x + 2 b} \right]^2  \nonumber\\
    & = & [12 (x + 2 b)^2 (x + 3 b)]^{- 1} \times \label{Gxyleq}\\
    &  & \Big[144 b^3 (b \xi + b - 1) + 24 b^2 (7 b \xi + 13 b - 11) x
    \nonumber\\
    &  &  + b (49 b \xi + 184 b - 144) x^2 + 8 (4 b - 3) x^3\Big] .
    \nonumber
  \end{eqnarray}
  From $\frac{1}{2} \leqslant \xi < b^{- 1} - 1$, we have
  \[ b \xi + b < 1, \quad 7 b \xi + 13 b < 11, \quad 49 b \xi + 184 b < 144
     \text{\quad and\quad} b < \frac{3}{4} . \]
  So the coefficients of $x$ on the right hand side of {\eqref{Gxyleq}} are
  all negative. Therefore, we obtain $\sup_{x, y \in [0, + \infty)} G (x, y) <
  0.$
\end{proof}

\begin{theorem}\label{codim 2-2}
  Let $F : M^n \times (- \infty, 0) \rightarrow \mathbb{S}^{n + p}$ $(p
  \geqslant 2)$ be a compact ancient solution of mean curvature flow in the
  unit sphere. Suppose
  \[ \sup_{M \times (- \infty, 0)} \left( | h |^2 - \frac{4}{3 n} | H |^2
     \right) < \frac{n}{\xi + 1}, \]
  where $\xi = \frac{2}{4 - \tmop{sgn} (p - 2)}$. Then $M_t$ is either a
  shrinking spherical cap or a totally geodesic sphere.
\end{theorem}

\begin{proof}
  From the pinching condition, there exists a positive constant $b$, such that
  for all $t < 0$ there holds
  \begin{equation}
    | \mathring{h} |^2 - \frac{1}{3 n} | H |^2 < b n < \frac{n}{\xi + 1} .
    \label{existb}
  \end{equation}

  We study the function
  \[ f = \frac{\normho^2}{| H |^2 + 2 b n^2} . \]
  It follows from {\eqref{existb}} that $f < \frac{1}{2 n}$.

  We have
  \[ \dt f = \frac{1}{| H |^2 + 2 b n^2} \left( \dt \normho^2 - f \dt | H |^2
     \right) \]
  and
  \[ \Delta f = \frac{1}{| H |^2 + 2 b n^2} \left( \Delta \normho^2 - f \Delta
     | H |^2 - 2 \langle \nabla f, \nabla | H |^2 \rangle \right) . \]
  From the evolution equations, we get
  \begin{eqnarray}
    &&\left( \dt - \Delta \right) f \nonumber\\& = & \frac{2}{| H |^2 + 2 b n^2} \langle
    \nabla f, \nabla | H |^2 \rangle - \frac{2}{| H |^2 + 2 b n^2} (| \nabla
    \mathring{h} |^2 - f | \nabla H |^2) \nonumber\\
    &  & + \frac{2}{| H |^2 + 2 b n^2} \left( R_1 - \frac{1}{n} R_2 \right) -
    2 f \left[ n + \frac{1}{| H |^2 + 2 b n^2} (R_2 + n | H |^2) \right] .
    \label{dtfinS}
  \end{eqnarray}
  Firstly, we estimate the gradient terms in (\ref{dtfinS}). From {\eqref{dh2}}
  we have
  \[ | \nabla \mathring{h} |^2 \geqslant \frac{1}{2 n} | \nabla H |^2 \geqslant f |
     \nabla H |^2 . \]
  Secondly, we estimate the reaction terms. From {\eqref{R1nogreater}},
  {\eqref{R2eq}} and the pinching condition, we have
  \begin{eqnarray*}
    R_1 - \frac{1}{n} R_2 & \leqslant & | h |^2 \normho^2 + 2 \normho^2 P -
    \frac{1}{n} |H|^2 P - \xi^{- 1} P^2\\
    & \leqslant & \normho^2 \left( | h |^2 + 2 P - \frac{\frac{1}{n} |H|^2
    P + \xi^{- 1} P^2}{\frac{1}{3 n} | H |^2 + b n} \right) .
  \end{eqnarray*}
  Hence we get
  \begin{eqnarray}
    &  & \frac{2}{| H |^2 + 2 b n^2} \left( R_1 - \frac{1}{n} R_2 \right) - 2
    f \left[ n + \frac{1}{| H |^2 + 2 b n^2} (R_2 + n | H |^2) \right]
    \label{2kH2bR1nR2}\\
    & \leqslant & 2 f \left[ | h |^2 + 2 P - \frac{\frac{1}{n} |H|^2 P +
    \xi^{- 1} P^2}{\frac{1}{3 n} | H |^2 + b n} - n - \frac{| H |^2 (| h |^2 -
    P + n)}{| H |^2 + 2 b n^2} \right]  \nonumber\\
    & = & 2 f \left[ \frac{2 b n^2 | h |^2 + | H |^2 (P - n)}{| H |^2 + 2 b
    n^2} - \frac{\frac{1}{n} |H|^2 P + \xi^{- 1} P^2}{\frac{1}{3 n} | H |^2
    + b n} + 2 P - n \right] \nonumber\\
    & \leqslant & 2 f \left[ \frac{2 b n^2 \left( \frac{4}{3 n} | H |^2 + b n
    \right) + | H |^2 (P - n)}{| H |^2 + 2 b n^2} - \frac{\frac{1}{n} |H|^2
    P + \xi^{- 1} P^2}{\frac{1}{3 n} | H |^2 + b n} + 2 P - n \right]
    \nonumber\\
    & = & 2 n f\, G \left( \frac{1}{n^2} | H |^2, \frac{1}{n} P \right) . \nonumber
  \end{eqnarray}
  Here $G$ is the function
  defined in Lemma \ref{Gxy}. Since $\frac{1}{2} \leqslant \xi < b^{- 1} - 1$,
  $G \left( \frac{1}{n^2} | H |^2, \frac{1}{n} P \right)$ has a negative upper
  bound $- \delta$.

  Inserting {\eqref{2kH2bR1nR2}} into (\ref{dtfinS}), we obtain
  \[ \left( \dt - \Delta \right) f \leqslant \frac{2}{| H |^2 + 2 b n^2} \langle
     \nabla f, \nabla | H |^2 \rangle - 2 n \delta f. \]
  Applying the maximum principle, we get
  \[ \forall s < t, \quad \max_{M_t} f \leqslant \mathe^{- 2 n \delta (t - s)}
     \max_{M_s} f. \]
  Letting $s \rightarrow - \infty$, we obtain $\max_{M_t} f = 0$, i.e., $M_t$
  is a totally umbilical sphere.
\end{proof}

If $p \geqslant 3$, we verify the following theorem with different pinching coefficients.
\begin{theorem}\label{codim 2-1}
  Let $F : M^n \times (- \infty, 0) \rightarrow \mathbb{S}^{n + p}$ $(p
  \geqslant 3)$ be a compact ancient solution of mean curvature flow in the
  unit sphere. If
  \[ \sup_{M \times (- \infty, 0)} \normho^2 < \frac{2 n}{3}, \]
  then $M_t$ is either a shrinking spherical cap or a totally geodesic sphere.
\end{theorem}

\begin{proof}
  Let
  \[ f = \frac{\normho^2}{| H |^2 + \frac{5}{3} n^2} . \]
  Then we have $f < \frac{2}{5 n}$.
  
  Similarly to {\eqref{dtfinS}}, we get
  \begin{eqnarray*}
    &&\left( \dt - \Delta \right) f \\& = & \frac{2}{| H |^2 + \frac{5}{3} n^2}
    \langle \nabla f, \nabla | H |^2 \rangle - \frac{2}{| H |^2 + \frac{5}{3}
    n^2} (| \nabla \mathring{h} |^2 - f | \nabla H |^2)\\
    &  & + \frac{2}{| H |^2 + \frac{5}{3} n^2} \left( R_1 - \frac{1}{n} R_2
    \right) - 2 f \left[ n + \frac{1}{| H |^2 + \frac{5}{3} n^2} (R_2 + n | H
    |^2) \right] .
  \end{eqnarray*}
  From {\eqref{dh2}} we have
  \[ | \nabla \mathring{h} |^2 \geqslant \frac{2}{5 n} | \nabla H |^2 \geqslant f |
     \nabla H |^2 . \]
  From {\eqref{R12R2/n}} and {\eqref{R2eq}}, we have
  \begin{eqnarray*}
    &  & \frac{2}{| H |^2 + \frac{5}{3} n^2} \left( R_1 - \frac{1}{n} R_2
    \right)\\
    & \leqslant & \frac{2}{| H |^2 + \frac{5}{3} n^2} \left( \frac{3}{2} |
    \mathring{h} |^4 + \frac{1}{n} | H |^2 | \mathring{h} |^2 - \frac{1}{n} | H |^2 P
    \right)\\
    & = & 2 f \left( \frac{1}{2} | \mathring{h} |^2 + | h |^2 \right) -
    \frac{2}{n} \cdot \frac{| H |^2 P}{| H |^2 + \frac{5}{3} n^2}
  \end{eqnarray*}
  and
  \begin{eqnarray*}
    &  & 2 f \left[ n + \frac{1}{| H |^2 + \frac{5}{3} n^2} (R_2 + n | H |^2)
    \right]\\
    & = & 2 f \left[ n + \frac{1}{| H |^2 + \frac{5}{3} n^2} | H |^2 (| h |^2
    + n - P) \right]\\
    & = & 2 f \left[ | h |^2 + 2 n - \frac{\frac{5}{3} n^2}{| H |^2 +
    \frac{5}{3} n^2} (| h |^2 + n) \right] - 2 f \frac{| H |^2 P}{| H |^2 +
    \frac{5}{3} n^2}\\
    & \geqslant & 2 f \left[ | h |^2 + 2 n - \frac{\frac{5}{3} n^2}{| H |^2 +
    \frac{5}{3} n^2} \left( \frac{1}{n} | H |^2 + \frac{5 n}{3} \right)
    \right] - \frac{2}{n} \cdot \frac{| H |^2 P}{| H |^2 + \frac{5}{3} n^2}\\
    & = & 2 f \left( | h |^2 + \frac{n}{3} \right) - \frac{2}{n} \cdot
    \frac{| H |^2 P}{| H |^2 + \frac{5}{3} n^2} .
  \end{eqnarray*}
  Therefore, we get
  \[ \left( \dt - \Delta \right) f \leqslant \frac{2}{| H |^2 + \frac{5}{3}
     n^2} \langle \nabla f, \nabla | H |^2 \rangle + \left( \normho^2 -
     \frac{2 n}{3} \right) f. \]
  Set $\delta = - \sup\limits_{^{M \times (- \infty, 0)}} \left( \normho^2 - \frac{2
  n}{3} \right)$. It follows from the maximum principle that
  \[ \forall s < t, \quad \max_{M_t} f \leqslant \mathe^{- \delta (t - s)}
     \max_{M_s} f. \]
  Hence, we obtain $\max_{M_t} f = 0$, i.e., $M_t$ is a totally umbilical
  sphere.
\end{proof}

Now we give the proof of Theorem \ref{thm-2}.

\begin{proof*}{Proof of Theorem \ref{thm-2}}
  By the assumption, there exists $t_0 < 0$, such that
  \[ \sup_{M \times (- \infty, t_0)} (| h |^2 - \kappa | H |^2) < \alpha,
  \]
  where $(\kappa, \alpha) = \big( \min \{ \tfrac{3}{n + 2}, \tfrac{4 (n - 1)}{n (n + 2)}
    \}, n \big)$ for $p = 1$,
    $(\kappa, \alpha) = ( \frac{4}{3 n}, \frac{2 n}{3} )$ for $p = 2$,
    $(\kappa, \alpha) = ( \frac{4}{3 n}, \frac{3 n}{5} ) \text{ or } (
    \frac{1}{n}, \frac{2 n}{3} )$ for $p \geqslant 3$.
  Then combining the results of Theorems \ref{codim 1 sphere}, \ref{codim 2-2} and \ref{codim 2-1}, we complete the proof of
  Theorem \ref{thm-2}.
\end{proof*}

\subsection{Ancient solutions in hyperbolic spaces}\

Let $F : M^n \times (- \infty, 0) \rightarrow \mathbb{H}^{n + p}$ be a
compact ancient solution of mean curvature flow in the hyperbolic space with
constant curvature $- 1$. Suppose there exists a positive number
$\varepsilon$, such that for all $t < 0$, $M_t$ satisfies $| H | > n$ and
\[ | \mathring{h} |^2 \leqslant k | H |^2  \left( 1 - \tfrac{n^2}{| H |^2}
   \right)^{2 + \varepsilon}, \]
where
\[ k = \begin{cases}
     \frac{1}{3 n}, & p \geqslant 2 \enspace \text{and} \enspace n \geqslant
     7,\\
     \frac{n - 1}{2 n (n + 2)}, & \text{otherwise} .
   \end{cases} \]

For positive numbers $\varepsilon$ and $n$, we define a function $\varphi :
(n^2, + \infty) \rightarrow \mathbb{R}$ by
\begin{equation}
  \varphi (x) = x \left( 1 - \frac{n^2}{x} \right)^{2 + \varepsilon} .
\end{equation}
Then it has the following properties.

\begin{lemma}
  \label{phy}The function $\varphi (x)$ satisfies
  \begin{enumerateroman}
    \item $\varphi (x) < x, \varphi' (x) < 1, \varphi'' (x) > 0$,

    \item $1 - \frac{x \varphi' (x)}{\varphi (x)} = - \frac{(2 + \varepsilon)
    n^2}{x - n^2}$,

    \item $\varphi' (x) + 2 x \varphi'' (x) < 4$.
  \end{enumerateroman}
\end{lemma}

\begin{proof}
  By the definition of $\varphi$, we get $\varphi (x) < x$.

  Differentiating $\varphi$, we get
  \[ \varphi' (x) = \frac{x + n^2 (1 + \varepsilon)}{x} \left( 1 -
     \frac{n^2}{x} \right)^{1 + \varepsilon}, \]
  \[ \varphi'' (x) = \frac{n^4 (1 + \varepsilon) (2 + \varepsilon)}{x^3}
     \left( 1 - \frac{n^2}{x} \right)^{\varepsilon} . \]
  Since $\varphi'' (x) > 0$, we have $\varphi' (x) < \varphi' (+ \infty) = 1$.

  Now we can check
  \[ 1 - \frac{x \varphi' (x)}{\varphi (x)} = - \frac{(2 + \varepsilon)
     n^2}{x - n^2} . \]

  By a direct computation, we have
  \[ \varphi' (x) + 2 x \varphi'' (x) = \frac{n^4 (1 + \varepsilon) (3 + 2
     \varepsilon) + n^2 \varepsilon x + x^2}{x^2} \left( 1 - \frac{n^2}{x}
     \right)^{\varepsilon} . \]
  Replacing $x$ by $\frac{n^2}{y}$, we set
  \[ \psi (y) = [(1 + \varepsilon) (3 + 2 \varepsilon) y^2 + \varepsilon y +
     1] (1 - y)^{\varepsilon} \]
  for $0 < y < 1$. Since
  \[ \psi' (y) = (1 + \varepsilon) (2 + \varepsilon) y (1 - y)^{\varepsilon - 1}
     [3 - (3 + 2 \varepsilon) y], \]
  we have
  \[ \sup_{0 < y < 1} \psi (y) = \psi \left( \frac{3}{3 + 2 \varepsilon}
     \right) = \frac{2 (7 \varepsilon + 6) }{2 \varepsilon + 3} \left( \frac{2
     \varepsilon}{2 \varepsilon + 3} \right)^{\varepsilon} . \]
  Taking the logarithm, we set
  \[ \theta (\varepsilon) = \log \frac{2 (7 \varepsilon + 6) }{2 \varepsilon +
     3} + \varepsilon \log \frac{2 \varepsilon}{2 \varepsilon + 3} . \]
  Then we have
  \[ \theta' (\varepsilon) = \frac{3 (7 \varepsilon + 9)}{(2 \varepsilon + 3)
     (7 \varepsilon + 6)} + \log \frac{2 \varepsilon}{2 \varepsilon + 3}, \]
  \[ \theta'' (\varepsilon) = \frac{27 (7 \varepsilon^2 + 17 \varepsilon +
     12)}{\varepsilon (2 \varepsilon + 3)^2 (7 \varepsilon + 6)^2} . \]
  Since $\theta'' (\varepsilon) > 0$, we have $\theta' (\varepsilon) < \theta'
  (+ \infty) = 0$. This implies $\theta (\varepsilon) < \theta (0 +) = \log
  4$. Thus we obtain $\psi (y) < 4$ for $0 < y < 1$.
\end{proof}

Now we give the proof of Theorem \ref{thm-hyperbolic}.

\begin{proof*}{Proof of Theorem \ref{thm-hyperbolic}}
We study the function
\[ f = \frac{\normho^2}{\varphi (| H |^2)} . \]
The pinching condition implies that $f \leqslant k$.

First, we derive the evolution equation of $f$. By a direct computation, we
have
\[ \dt f = \frac{1}{\varphi (| H |^2)} \left( \dt \normho^2 - f \dt \varphi (|
   H |^2) \right) \]
and
\[ \Delta f = \frac{1}{\varphi (| H |^2)} \left( \Delta \normho^2 - f \Delta
   \varphi (| H |^2) - 2 \langle \nabla f, \nabla \varphi (| H |^2) \rangle
   \right) . \]
From Lemma \ref{evo} (ii), we have
\begin{eqnarray*}
  &&\left( \dt - \Delta \right) \varphi (| H |^2) \\ & = & - 2 \varphi' (| H |^2)
  (| \nabla H |^2 - R_2 + n | H |^2) - \varphi'' (| H |^2) \big| \nabla | H |^2
  \big|^2\\
  & \geqslant & - 2 (\varphi' (| H |^2) + 2 | H |^2 \varphi'' (| H |^2)) |
  \nabla H |^2\\
  &  & + 2 \varphi' (| H |^2)  (R_2 - n | H |^2) .
\end{eqnarray*}
Therefore we obtain
\begin{eqnarray}
  &&\left( \dt - \Delta \right) f \nonumber\\& \leqslant & \frac{2}{\varphi (| H |^2)}
  \langle \nabla f, \nabla \varphi (| H |^2) \rangle  \label{dtfinH}\\
  &  & - \frac{2}{\varphi (| H |^2)} [| \nabla \mathring{h} |^2 - f (\varphi' (| H
  |^2) + 2 | H |^2 \varphi'' (| H |^2)) | \nabla H |^2] \nonumber\\
  &  & + \frac{2}{\varphi (| H |^2)} \left( R_1 - \frac{1}{n} R_2 \right) + 2
  f \left[ n - \frac{\varphi' (| H |^2)}{\varphi (| H |^2)} (R_2 - n | H |^2)
  \right] . \nonumber
\end{eqnarray}

From {\eqref{dh2}} and Lemma \ref{phy} (iii), we have
\begin{eqnarray*}
  &  & f (\varphi' (| H |^2) + 2 | H |^2 \varphi'' (| H |^2)) | \nabla H
  |^2\\
  &  & \leqslant 4 k | \nabla H |^2 \leqslant \frac{2 (n - 1)}{n (n + 2)} |
  \nabla H |^2 \leqslant | \nabla \mathring{h} |^2 .
\end{eqnarray*}
From {\eqref{R1nogreater}}, {\eqref{R2eq}}, we get
\begin{eqnarray*}
  &  & \frac{2}{\varphi (| H |^2)} \left( R_1 - \frac{1}{n} R_2 \right) + 2 f
  \left[ n - \frac{\varphi' (| H |^2)}{\varphi (| H |^2)} (R_2 - n | H |^2)
  \right]\\
  & \leqslant & \frac{2}{\varphi (| H |^2)} \left[ | h |^2 \normho^2 + 2 P
  \normho^2 - \frac{1}{n} P | H |^2 \right] \\ &&+ 2 f \left[ n - \frac{\varphi' (|
  H |^2) | H |^2}{\varphi (| H |^2)} (| h |^2 - P - n) \right]\\
  & = & 2 f \left[ | h |^2 + n - \frac{\varphi' (| H |^2) | H |^2}{\varphi (|
  H |^2)} (| h |^2 - n) \right] \\ &&+ \frac{2 P}{\varphi (| H |^2)} \left[ 2
  \normho^2 - \frac{1}{n} | H |^2 + f \varphi' (| H |^2) | H |^2 \right] .
\end{eqnarray*}
If $p = 1$, then $P$ is identically zero. If $p \geqslant 2$, it follows from
the pinching condition and Lemma \ref{phy} (i) that
\begin{eqnarray*}
  &  & 2 \normho^2 - \frac{1}{n} | H |^2 + f \varphi' (| H |^2) | H |^2\\
  & \leqslant & 2 k \varphi (| H |^2) - \frac{1}{n} | H |^2 + f | H |^2\\
  & \leqslant & 2 k | H |^2 - \frac{1}{n} | H |^2 + k | H |^2\\
  & = & \left( 3 k - \frac{1}{n} \right) | H |^2\\
  & \leqslant & 0.
\end{eqnarray*}
By Lemma \ref{phy} (ii), we have
\begin{eqnarray*}
  &  & | h |^2 + n - \frac{\varphi' (| H |^2) | H |^2}{\varphi (| H |^2)} (|
  h |^2 - n)\\
  & = & \left( 1 - \frac{\varphi' (| H |^2) | H |^2}{\varphi (| H |^2)}
  \right) (| h |^2 - n) + 2 n\\
  & = & - \frac{(2 + \varepsilon) n}{| H |^2 - n^2} (n | h |^2 - n^2) + 2 n\\
  & \leqslant & - (2 + \varepsilon) n + 2 n\\
  & = & - \varepsilon n.
\end{eqnarray*}
Hence, we obtain
\[ \left( \dt - \Delta \right) f \leqslant \frac{2}{\varphi (| H |^2)} \langle
   \nabla f, \nabla \varphi (| H |^2) \rangle - 2 \varepsilon n f. \]
This implies $f \equiv 0$, i.e., $M_t$ is a totally umbilical sphere.
\end{proof*}

\section{Integral pinched ancient solutions}

\subsection{Ancient solutions in the Euclidean space}\

We need the following Sobolev type inequality on submanifolds \cite{MS}.

\begin{proposition}
  \label{Sobo}Let $M$ be an $n$-dimensional closed submanifold in the
  Euclidean space. For any nonnegative $C^1$-function $f$ on $M$, we have
  \[ \left( \int_M f^{\frac{n}{n - 1}} \mathd \mu \right)^{\frac{n - 1}{n}}
     \leqslant B (n) \int_M (| \nabla f | + f | H |) \mathd \mu, \]
  where $B (n)$ is a positive constant depending only on $n$.
\end{proposition}

Now we investigate the integral pinched ancient solution in dimension two.

\begin{theorem}
  Let $F : M^2 \times (- \infty, 0) \rightarrow \mathbb{R}^{2 + p}$ be a
  compact ancient solution of mean curvature flow. Suppose for all $t < 0$,
  there holds
  \[ \int_{M_t} | \mathring{h} |^2 \mathd \mu_t < C, \]
  where $C$ is an explicit positive constant. Then $M_t$ is a shrinking
  sphere.
\end{theorem}

\begin{proof}
  Let $\chi (M)$ be the Euler characteristic of $M$. The Gauss-Bonnet formula
  implies
  \begin{equation}
    \int_{M_t} \left( \frac{1}{4} | H |^2 - \frac{1}{2} | \mathring{h} |^2 \right)
    \mathd \mu_t = 2 \pi \chi (M) . \label{GBform}
  \end{equation}
  Noting that $\dt \mathd \mu_t = - | H |^2 \mathd \mu_t$, we have
  \begin{eqnarray*}
    \frac{\mathd}{\mathd t} \int_{M_t} | \mathring{h} |^2 \mathd \mu_t & = &
    \frac{\mathd}{\mathd t} \int_{M_t} \left( \frac{3}{2} | \mathring{h} |^2 -
    \frac{1}{4} | H |^2 \right) \mathd \mu_t\\
    & = & \int_{M_t} \left[ \frac{\partial}{\partial t} \left( \frac{3}{2} |
    \mathring{h} |^2 - \frac{1}{4} | H |^2 \right) - | H |^2 \left( \frac{3}{2} |
    \mathring{h} |^2 - \frac{1}{4} | H |^2 \right) \right] \mathd \mu_t\\
    & = & \int_{M_t} \left[ - 3 | \nabla \mathring{h} |^2 + \frac{1}{2} | \nabla H
    |^2 + 3 R_1 - 2 R_2 \right.\\
    &  & \left. - | H |^2 \left( \frac{3}{2} | \mathring{h} |^2 - \frac{1}{4} | H
    |^2 \right) \right] \mathd \mu_t .
  \end{eqnarray*}
  By {\eqref{dh2}}, we have
  \begin{equation}
    - 3 | \nabla \mathring{h} |^2 + \frac{1}{2} | \nabla H |^2 \leqslant - | \nabla
    \mathring{h} |^2 . \label{3delho2}
  \end{equation}
  By {\eqref{R1nogreater}} and {\eqref{R2eq}}, we have
  \begin{eqnarray}
    3 R_1 - 2 R_2 & \leqslant & 3 | h |^4 + 3 (2 | \mathring{h} |^2 - | H |^2) P -
    2 | H |^2 (| h |^2 - P) \nonumber\\
    & \leqslant & 3 | h |^4 + 6 | \mathring{h} |^4 - 2 | H |^2 | h |^2
    \label{3R12R2}\\
    & = & 9 | \mathring{h} |^4 + | H |^2 | \mathring{h} |^2 - \frac{1}{4} | H |^4 .
    \nonumber
  \end{eqnarray}
  So we get
  \[ \frac{\mathd}{\mathd t} \int_{M_t} | \mathring{h} |^2 \mathd \mu_t \leqslant
     \int_{M_t} \left( 9 | \mathring{h} |^4 - | \nabla \mathring{h} |^2 - \frac{1}{2} |
     H |^2 | \mathring{h} |^2 \right) \mathd \mu_t . \]
  From Proposition \ref{Sobo}, we have
  \begin{eqnarray*}
    \left( \int_{M_t} | \mathring{h} |^4 \mathd \mu_t \right)^{\frac{1}{2}} &
    \leqslant & B (2) \int_{M_t} (2 | \mathring{h} | | \nabla \mathring{h} | + | \mathring{h}
    |^2 | H |) \mathd \mu_t\\
    & \leqslant & B (2) \left( \int_{M_t} | \mathring{h} |^2 \mathd \mu_t
    \right)^{\frac{1}{2}} \left[ \int_{M_t} (2 | \nabla \mathring{h} | + | \mathring{h}
    |^{} | H |)^2 \mathd \mu_t \right]^{\frac{1}{2}}\\
    & \leqslant & B (2) \sqrt{C} \left[ \int_{M_t} (6 | \nabla \mathring{h} |^2 +
    3 | \mathring{h} |^2 | H |^2) \mathd \mu_t \right]^{\frac{1}{2}} .
  \end{eqnarray*}
  Setting $C = \frac{1}{60 B (2)^2}$, we get
  \[  \int_{M_t} | \mathring{h} |^4 \mathd \mu_t  \leqslant
     \frac{1}{10} \int_{M_t} \left( | \nabla \mathring{h} |^2 + \frac{1}{2} |
     \mathring{h} |^2 | H |^2 \right) \mathd \mu_t . \]
  Hence we obtain
  \begin{equation}
    \frac{\mathd}{\mathd t} \int_{M_t} | \mathring{h} |^2 \mathd \mu_t \leqslant -
    \int_{M_t} | \mathring{h} |^4 \mathd \mu_t \leqslant - \frac{\left( \int_{M_t} |
    \mathring{h} |^2 \mathd \mu_t \right)^2}{\int_{M_t} 1 \mathd \mu_t} .
    \label{dtho2norm}
  \end{equation}

  Let $I (t) = \int_{M_t} | \mathring{h} |^2 \mathd \mu_t$, $\mathit{vol} (t) =
  \int_{M_t} 1 \mathd \mu_t$. We prove $I (t)$ is always zero by contradiction.
  Assume there exists $t_0$, such that $I (t_0) > 0$. Thus, for all $t < t_0$,
  we get from {\eqref{dtho2norm}} that
  \[ \frac{\mathd}{\mathd t} I^{- 1} (t) \geqslant \mathit{vol}^{- 1} (t) . \]
  Integrating, we get
  \begin{equation}
    I^{- 1} (t_0) - I^{- 1} (t) \geqslant \int_t^{t_0} \mathit{vol}^{- 1} (\tau)
    \mathd \tau . \label{ho2norm}
  \end{equation}
  Then we estimate the volume of $M_t$. By {\eqref{GBform}}, we have
  \[ \frac{\mathd}{\mathd t} \mathit{vol} (t) = - \int_{M_t} | H |^2 \mathd
     \mu_t \geqslant - \bar{C} , \]
  where $\bar{C} = 2 C + 16 \pi$. This yields
  \[ \mathit{vol} (t) \leqslant \mathit{vol} (t_0) + \bar{C} (t_0 -
     t) . \]
  Thus, we get from {\eqref{ho2norm}} that
  \[ I^{- 1} (t_0) \geqslant I^{- 1} (t) + \frac{1}{\bar{C}} \log
     \left[ 1 + \frac{\bar{C}}{\mathit{vol} (t_0)} (t_0 - t) \right] .
  \]
  The right hand side of the above inequality tends to $+ \infty$ as $t
  \rightarrow - \infty$, which leads to a contradiction. Therefore, we obtain $I
  (t) = 0$ for all $t$.
\end{proof}

Denote by $f_+$ the positive part of a function $f$. Now we derive the
following integral inequality.

\begin{lemma}
  \label{intULapU}Suppose $M$ is an $n (\geqslant 3)$-dimensional closed
  submanifold in the Euclidean space, and $U$ is a $C^2$-function on $M$. For
  any $q > 1$, we have
  \[ \int_M U_+^{q - 1} \Delta U \mathd \mu \leqslant - \frac{q - 1}{4 q^2}
     \left[ \frac{1}{2 B (n)^2} \left( \int_M U_+^{\frac{q n}{n - 2}} \mathd
     \mu \right)^{\frac{n - 2}{n}} - \int_M | H |^2 U_+^q \mathd \mu \right] .
  \]
\end{lemma}

\begin{proof}
  Set $U_{\varepsilon} = \sqrt{U_+^2 + \varepsilon}$ for $\varepsilon > 0$.
  Then $U_{\varepsilon}$ is $C^1$-differentiable, and $\nabla U_{\varepsilon}
  = \frac{U_+}{U_{\varepsilon}} \nabla U$. By the divergence theorem, we have
  \begin{eqnarray}
    \int_M U_{\varepsilon}^{q - 1} \Delta U \mathd \mu &=& - (q - 1) \int_M
    U_{\varepsilon}^{q - 3} U_+ | \nabla U |^2 \mathd \mu \nonumber\\ &\leqslant& - (q - 1)
    \int_M U_{\varepsilon}^{q - 2} | \nabla U_{\varepsilon} |^2 \mathd \mu .
    \label{intUeq}
  \end{eqnarray}
  Applying Proposition \ref{Sobo} to $U_{\varepsilon}^{q (n - 1) / (n - 2)}$,
  we have
  \begin{eqnarray*}
    \left( \int_M U_{\varepsilon}^{\frac{q n}{n - 2}} \mathd \mu
    \right)^{\frac{n - 1}{n}} & \leqslant & B (n) \int_M \left( \left| \nabla
    U_{\varepsilon}^{\frac{q (n - 1)}{n - 2}} \right| +
    U_{\varepsilon}^{\frac{q (n - 1)}{n - 2}} | H | \right) \mathd \mu\\
    & \leqslant & B (n) \int_M U_{\varepsilon}^{\frac{q n}{2 (n - 2)}} \left(
    2 q U_{\varepsilon}^{\frac{q}{2} - 1} | \nabla U_{\varepsilon} | +
    U_{\varepsilon}^{\frac{q}{2}} | H | \right) \mathd \mu\\
    & \leqslant & B (n) \left( \int_M U_{\varepsilon}^{\frac{q n}{n - 2}}
    \right)^{\frac{1}{2}} \left[ \int_M \left( 2 q
    U_{\varepsilon}^{\frac{q}{2} - 1} | \nabla U_{\varepsilon} | +
    U_{\varepsilon}^{\frac{q}{2}} | H | \right)^2 \mathd \mu
    \right]^{\frac{1}{2}} .
  \end{eqnarray*}
  This implies
  \begin{equation}
    \left( \int_M U_{\varepsilon}^{\frac{q n}{n - 2}} \mathd \mu
    \right)^{\frac{n - 2}{n}} \leqslant B (n)^2 \int_M 2 (4 q^2
    U_{\varepsilon}^{q - 2} | \nabla U_{\varepsilon} |^2 + U_{\varepsilon}^q |
    H |^2) \mathd \mu . \label{intUeqn}
  \end{equation}
  Combining {\eqref{intUeq}} and {\eqref{intUeqn}}, we get
  \[ \int_M U_{\varepsilon}^{q - 1} \Delta U \mathd \mu \leqslant - \frac{q -
     1}{4 q^2} \left[ \frac{1}{2 B (n)^2} \left( \int_M
     U_{\varepsilon}^{\frac{q n}{n - 2}} \mathd \mu \right)^{\frac{n - 2}{n}}
     - \int_M | H |^2 U_{\varepsilon}^q \mathd \mu \right] . \]
  Letting $\varepsilon \rightarrow 0$, we obtain
  \[ \int_M U_+^{q - 1} \Delta U \mathd \mu \leqslant - \frac{q - 1}{4 q^2}
     \left[ \frac{1}{2 B (n)^2} \left( \int_M U_+^{\frac{q n}{n - 2}} \mathd
     \mu \right)^{\frac{n - 2}{n}} - \int_M | H |^2 U_+^q \mathd \mu \right] .
  \]

\end{proof}

Let $n \geqslant 3$. Consider the compact ancient solution $F : M^n \times (-
\infty, 0) \rightarrow \mathbb{R}^{n + p}$. Set
\[ U = | \mathring{h} |^2 - \frac{1}{n^2} | H |^2 . \]
Then we have the following estimate.

\begin{lemma}
  \label{dtUposq}For any number $q > 1$, the following inequality holds along
  the flow.
  \[ \frac{\mathd}{\mathd t} \int_{M_t} U_+^q \mathd \mu_t \leqslant \left[ -
     A_1 (n, q) + A_2 (n, q) \left( \int_{M_t} | \mathring{h} |^n \mathd \mu_t
     \right)^{\frac{2}{n}} \right] \left( \int_{M_t} U_+^{\frac{q n}{n - 2}}
     \mathd \mu_t \right)^{\frac{n - 2}{n}}, \]
  where $A_1 (n, q), A_2 (n, q)$ are positive constants depending only on $n$ and $p$.
\end{lemma}

\begin{proof}
  From the evolution equations, we get
  \[ \left( \dt - \Delta \right) U = - 2 \left( | \nabla \mathring{h} |^2 -
     \frac{1}{n^2} | \nabla H |^2 \right) + 2 \left( R_1 - \frac{n + 1}{n^2}
     R_2 \right) . \]
  By {\eqref{dh2}} we have
  \begin{equation}
    | \nabla \mathring{h} |^2 - \frac{1}{n^2} | \nabla H |^2 \geqslant \frac{1}{3 n}
    | \nabla H |^2 . \label{dho2-n-2dH2}
  \end{equation}
  From {\eqref{R1nogreater}}, {\eqref{R2eq}}, we get
  \begin{eqnarray}
    &  & R_1 - \frac{n + 1}{n^2} R_2 \nonumber\\
    & \leqslant & | h |^4 + \left( 2 \normho^2 - \frac{2}{n} | H |^2 \right) P
    - \frac{n + 1}{n^2} | H |^2 (| h |^2 - P) \nonumber\\
    & \leqslant & U (| h |^2 + 2 P)  \label{R1n+1n-2R2}\\
    & \leqslant & U_+ (| h |^2 + 2 | \mathring{h} |^2)  \nonumber\\
    & = & U_+ \left( 3 | \mathring{h} |^2 + \frac{1}{n} | H |^2 \right) .\nonumber
  \end{eqnarray}
  Hence we obtain
  \begin{equation}
    \dt U \leqslant \Delta U - \frac{2}{3 n} | \nabla H |^2 + 2 U_+ \left( 3
    | \mathring{h} |^2 + \frac{1}{n} | H |^2 \right) . \label{dtUless}
  \end{equation}
  It follows from {\eqref{dtUless}} and Lemma \ref{intULapU} that
  \begin{eqnarray*}
    \frac{\mathd}{\mathd t} \int_{M_t} U_+^q \mathd \mu_t & \leqslant &
    \int_{M_t} q U_+^{q - 1} \frac{\partial}{\partial t} U \mathd \mu_t\\
    & \leqslant & q \int_{M_t} U_+^{q - 1} \Delta U \mathd \mu_t + 2 q
    \int_{M_t} U_+^q \left( 3 | \mathring{h} |^2 + \frac{1}{n} | H |^2 \right) \mathd \mu_t\\
    & \leqslant & - \frac{q - 1}{4 q} \cdot \frac{1}{2 B (n)^2} \left(
    \int_{M_t} U_+^{\frac{q n}{n - 2}} \mathd \mu_t \right)^{\frac{n -
    2}{n}}\\
    &  & + \int_{M_t} \left( \frac{q - 1}{4 q} | H |^2 + 2 q \left( 3
    | \mathring{h} |^2 + \frac{1}{n} | H |^2 \right) \right) U_+^q \mathd \mu_t .
  \end{eqnarray*}
  Note that $| H |^2 \leqslant n^2 | \mathring{h} |^2$ on the support of $U_+$. We
  use H{\"o}lder's inequality to get
  \begin{eqnarray*}
    &  & \int_{M_t} \left( \frac{q - 1}{4 q} | H |^2 + 2 q \left( 3
    | \mathring{h} |^2 + \frac{1}{n} | H |^2 \right) \right) U_+^q \mathd \mu_t\\
    & \leqslant & \left( \frac{n^2}{4} + 4 q n \right) \int_{M_t} | \mathring{h}
    |^2 U_+^q \mathd \mu_t\\
    & \leqslant & \left( \frac{n^2}{4} + 4 q n \right) \left( \int_{M_t} |
    \mathring{h} |^n \mathd \mu_t \right)^{\frac{2}{n}} \left( \int_{M_t}
    U_+^{\frac{q n}{n - 2}} \mathd \mu_t \right)^{\frac{n - 2}{n}} .
  \end{eqnarray*}
  Thus, we complete the proof.
\end{proof}

\begin{theorem}
  Let $F : M^n \times (- \infty, 0) \rightarrow \mathbb{R}^{n + p}$ $(n
  \geqslant 3)$ be a compact ancient solution of mean curvature flow. Suppose
  for all $t < 0$, there holds
  \[ \int_{M_t} | \mathring{h} |^n \mathd \mu_t < C (n), \]
  where $C (n)$ is a positive constant explicitly depending on $n$. Then $M_t$
  is a shrinking sphere.
\end{theorem}

\begin{proof}
  We choose $C (n)$ such that
  \[ C (n) < \min_{q \in \{ \frac{n}{2}, \frac{n^2}{2 (n - 2)} \}} \left(
     \frac{A_1 (n, q)}{A_2 (n, q)} \right)^{\frac{n}{2}} . \]
  Here $A_1 (n, q), A_2 (n, q)$ are constants in Lemma \ref{dtUposq}. Letting
  $J_r (t) = \int_{M_t} U_+^{r / 2} \mathd \mu_t$, we get
  \[ \frac{\mathd}{\mathd t} J_n (t) \leqslant - \tilde{C} (n) (J_{n^2 / (n -
     2)} (t))^{\frac{n - 2}{n}}, \]
  where $\tilde{C} (n)$ is a positive constant. Hence, for any $t_1 < t_2 <
  0$, we have
  \[ J_n (t_1) - J_n (t_2) \geqslant \tilde{C} (n) \int_{t_1}^{t_2} (J_{n^2 /
     (n - 2)} (t))^{\frac{n - 2}{n}} \mathd t. \]
  It follows from Lemma \ref{dtUposq} that $J_{n^2 / (n - 2)} (t)$ is
  decreasing. Thus
  \[ \int_{t_1}^{t_2} (J_{n^2 / (n - 2)} (t))^{\frac{n - 2}{n}} \mathd t
     \geqslant (t_2 - t_1) (J_{n^2 / (n - 2)} (t_2))^{\frac{n - 2}{n}} . \]
  On the other hand, we have
  \[ J_n (t_1) - J_n (t_2) \leqslant J_n (t_1) \leqslant \int_{M_{t_1}} |
     \mathring{h} |^n \mathd \mu_{t_1} < C (n) . \]
  Hence we get
  \[ (t_2 - t_1) (J_{n^2 / (n - 2)} (t_2))^{\frac{n - 2}{n}} < \frac{C
     (n)}{\tilde{C} (n)} . \]
  Letting $t_1 \rightarrow - \infty$, we obtain $J_{n^2 / (n - 2)} (t_2) = 0$.
  Thus $U \leqslant 0$ for all $t$.

  From {\eqref{dtUless}} we get
  \[ \dt U \leqslant \Delta U - \frac{2}{3 n} | \nabla H |^2 . \]
  Assume $| H |$ attains 0 at $(x_0, t_0)$. Then $U$ also attains 0 at this
  point. The strong maximum principle implies that $U \equiv 0$. Then we get $|
  \nabla H | \equiv 0$. Thus, $M_{t_0}$ is minimal, which is not possible. So
  $H$ is non-vanishing along the flow. Applying the rigidity theorem for ancient
  solutions in \cite{LN2017}, we obtain the conclusion.
\end{proof}

\subsection{Ancient solutions in the sphere}\

First we investigate the integral pinched ancient solution in dimension two.

\begin{theorem}
  Let $F : M^2 \times (- \infty, 0) \rightarrow \mathbb{S}^{2 + p}$ be a
  compact ancient solution of mean curvature flow. Suppose for all $t < 0$,
  there holds
  \[ \int_{M_t} | \mathring{h} |^2 \mathd \mu_t < C, \]
  where $C$ is an explicit positive constant. Then $M_t$ is either a shrinking
  spherical cap or a totally geodesic sphere.
\end{theorem}

\begin{proof}
  The Gauss-Bonnet formula  implies
  \[ \int_{M_t} \left( \frac{1}{4} | H |^2 - \frac{1}{2} | \mathring{h} |^2 + 1
     \right) \mathd \mu_t = 2 \pi \chi (M) . \]
  Noting that $\dt \mathd \mu_t = - | H |^2 \mathd \mu_t$, we have
  \begin{eqnarray*}
    \frac{\mathd}{\mathd t} \int_{M_t} | \mathring{h} |^2 \mathd \mu_t & = &
    \frac{\mathd}{\mathd t} \int_{M_t} \left( \frac{3}{2} | \mathring{h} |^2 -
    \frac{1}{4} | H |^2 - 1 \right) \mathd \mu_t\\
    & = & \int_{M_t} \left[ \frac{\partial}{\partial t} \left( \frac{3}{2} |
    \mathring{h} |^2 - \frac{1}{4} | H |^2 \right) - | H |^2 \left( \frac{3}{2} |
    \mathring{h} |^2 - \frac{1}{4} | H |^2 - 1 \right) \right] \mathd \mu_t\\
    & = & \int_{M_t} \left[ - 3 | \nabla \mathring{h} |^2 + \frac{1}{2} | \nabla H
    |^2 + 3 R_1 - 2 R_2 \right.\\
    &  & \left. - 6 | \mathring{h} |^2 - | H |^2 \left( \frac{3}{2} | \mathring{h} |^2
    - \frac{1}{4} | H |^2 \right) \right] \mathd \mu_t .
  \end{eqnarray*}
  Using {\eqref{3delho2}} and {\eqref{3R12R2}} again, we get
  \[ \frac{\mathd}{\mathd t} \int_{M_t} | \mathring{h} |^2 \mathd \mu_t \leqslant
     \int_{M_t} \left( 9 | \mathring{h} |^4 - | \nabla \mathring{h} |^2 - \frac{1}{2} |
     H |^2 | \mathring{h} |^2 - 6 | \mathring{h} |^2 \right) \mathd \mu_t . \]
  Through the composition of immersions $M^n \rightarrow \mathbb{S}^{n + p}
  \rightarrow \mathbb{R}^{n + p + 1}$, $M^n$ can be regarded as a submanifold
  in $\mathbb{R}^{n + p + 1}$, whose mean curvature is $\sqrt{| H |^2 + n^2}$.
  Thus Proposition \ref{Sobo} implies
  \begin{eqnarray*}
    \left( \int_{M_t} | \mathring{h} |^4 \mathd \mu_t \right)^{\frac{1}{2}} &
    \leqslant & B (2) \int_{M_t} \left( 2 | \mathring{h} | | \nabla \mathring{h} | + |
    \mathring{h} |^2 \sqrt{| H |^2 + 4} \right) \mathd \mu_t\\
    & \leqslant & B (2) \left( \int_{M_t} | \mathring{h} |^2 \mathd \mu_t
    \right)^{\frac{1}{2}} \left[ \int_{M_t} \left( 2 | \nabla \mathring{h} | + |
    \mathring{h} |^{} \sqrt{| H |^2 + 4} \right)^2 \mathd \mu_t
    \right]^{\frac{1}{2}}\\
    & \leqslant & B (2) \sqrt{C} \left[ \int_{M_t} (6 | \nabla \mathring{h} |^2 +
    3 | \mathring{h} |^2 (| H |^2 + 4)) \mathd \mu_t \right]^{\frac{1}{2}} .
  \end{eqnarray*}
  Setting $C = \frac{1}{54 B (2)^2}$, we get
  \[  \int_{M_t} | \mathring{h} |^4 \mathd \mu_t  \leqslant
     \frac{1}{9} \int_{M_t} \left( | \nabla \mathring{h} |^2 + \frac{1}{2} |
     \mathring{h} |^2 | H |^2 + 2 | \mathring{h} |^2 \right) \mathd \mu_t . \]
  Hence we have
  \[ \frac{\mathd}{\mathd t} \int_{M_t} | \mathring{h} |^2 \mathd \mu_t \leqslant -
     4 \int_{M_t} | \mathring{h} |^2 \mathd \mu_t . \]
  Therefore, we obtain $\int_{M_t} | \mathring{h} |^2 \mathd \mu_t = 0$ for all
  $t$.
\end{proof}

Now we let $n \geqslant 3$. Consider the compact ancient solution $F : M^n
\times (- \infty, 0) \rightarrow \mathbb{S}^{n + p}$. Set
\[ U = | \mathring{h} |^2 - \frac{1}{n^2} | H |^2 . \]
We have the following estimate.

\begin{lemma}
  \label{dtUposqinS}Along the mean curvature flow, there holds
  \begin{eqnarray*}
    \frac{\mathd}{\mathd t} \int_{M_t} U_+^{\frac{n}{2}} \mathd \mu_t &
    \leqslant & - D_1 (n) \int_{M_t} U_+^{\frac{n}{2}} \mathd \mu_t\\
    &  & - \left[ D_2 (n) - D_3 (n) \left( \int_{M_t} | \mathring{h} |^n \mathd
    \mu_t \right)^{\frac{2}{n}} \right] \left( \int_{M_t} U_+^{\frac{n^2}{2 (n
    - 2)}} \mathd \mu_t \right)^{\frac{n - 2}{n}},
  \end{eqnarray*}
  where $D_1 (n), D_2 (n), D_3 (n)$ are positive constants depending only on $n$.
\end{lemma}

\begin{proof}
  Using {\eqref{dho2-n-2dH2}} and {\eqref{R1n+1n-2R2}} again, we have
  \begin{eqnarray*}
    &&\left( \dt - \Delta \right) U \\ & = & - 2 \left( | \nabla \mathring{h} |^2 -
    \frac{1}{n^2} | \nabla H |^2 \right) + 2 \left( R_1 - \frac{n + 1}{n^2}
    R_2 \right) - 2 n \left( \normho^2 + \frac{1}{n^2} | H |^2 \right)\\
    & \leqslant & 2 U_+ \left( 3 | \mathring{h} |^2 + \frac{1}{n} | H |^2 \right) - 2 n \left(
      \normho^2 + \frac{1}{n^2} | H |^2 \right)\\
    & \leqslant & 2 U_+ \left( 3 | \mathring{h} |^2 + \frac{1}{n} | H |^2 - n
    \right) .
  \end{eqnarray*}
  Hence we get
  \begin{eqnarray*}
    \frac{\mathd}{\mathd t} \int_{M_t} U_+^{\frac{n}{2}} \mathd \mu_t & = &
    \int_{M_t} \left( \frac{n}{2} U_+^{\frac{n}{2} - 1}
    \frac{\partial}{\partial t} U - | H |^2 U_+^{\frac{n}{2}} \right) \mathd
    \mu_t\\
    & \leqslant & \int_{M_t} \left( \frac{n}{2} U_+^{\frac{n}{2} - 1} \Delta
    U + n U_+^{\frac{n}{2}} (3 | \mathring{h} |^2 - n) \right) \mathd \mu_t .
  \end{eqnarray*}
  Regarding $M^n$ as a submanifold in $\mathbb{R}^{n + p + 1}$ with mean
  curvature $\sqrt{| H |^2 + n^2}$, we get from Lemma \ref{intULapU} that
  \[ \int_{M_t} U_+^{\frac{n}{2} - 1} \Delta U \mathd \mu_t \leqslant -
     \frac{n - 2}{4 n^2 B (n)^2} \left( \int_{M_t} U_+^{\frac{n^2}{2 (n - 2)}}
     \mathd \mu_t \right)^{\frac{n - 2}{n}} + \frac{n - 2}{2 n^2} \int_{M_t}
     (| H |^2 + n^2) U_+^{\frac{n}{2}} \mathd \mu_t . \]
  Thus we get
  \begin{eqnarray*}
    \frac{\mathd}{\mathd t} \int_{M_t} U_+^{\frac{n}{2}} \mathd \mu_t &
    \leqslant & - \frac{n - 2}{8 n B (n)^2} \left( \int_{M_t}
    U_+^{\frac{n^2}{2 (n - 2)}} \mathd \mu_t \right)^{\frac{n - 2}{n}}\\
    &  & + \frac{n - 2}{4 n} \int_{M_t} (| H |^2 + n^2) U_+^{\frac{n}{2}}
    \mathd \mu_t + \int_{M_t} n U_+^{\frac{n}{2}} (3 | \mathring{h} |^2 - n) \mathd
    \mu_t\\
    & \leqslant & - \frac{n - 2}{8 n B (n)^2} \left( \int_{M_t}
    U_+^{\frac{n^2}{2 (n - 2)}} \mathd \mu_t \right)^{\frac{n - 2}{n}}\\
    &  & + \frac{n^2 + 10 n}{4} \int_{M_t} | \mathring{h} |^2 U_+^{\frac{n}{2}}
    \mathd \mu_t - \frac{3 n^2 + 2 n}{4} \int_{M_t} U_+^{\frac{n}{2}} \mathd
    \mu_t .
  \end{eqnarray*}
  Using the following H{\"o}lder inequality
  \[ \int_{M_t} | \mathring{h} |^2 U_+^{\frac{n}{2}} \mathd \mu_t \leqslant \left(
     \int_{M_t} | \mathring{h} |^n \mathd \mu_t \right)^{\frac{2}{n}} \left(
     \int_{M_t} U_+^{\frac{n^2}{2 (n - 2)}} \mathd \mu_t \right)^{\frac{n -
     2}{n}}, \]
  we complete the proof.
\end{proof}

\begin{theorem}
  Let $F : M^n \times (- \infty, 0) \rightarrow \mathbb{S}^{n + p}$ $(n
  \geqslant 3)$ be a compact ancient solution of mean curvature flow. Suppose
  for all $t < 0$, there holds
  \[ \int_{M_t} | \mathring{h} |^n \mathd \mu_t < C (n), \]
  where $C (n)$ is a positive constant explicitly depending on $n$. Then $M_t$
  is either a shrinking spherical cap or a totally geodesic sphere.
\end{theorem}

\begin{proof}
  By choosing $C (n) = \left( \frac{D_2 (n)}{D_3 (n)} \right)^{n / 2}$, we get
  from Lemma \ref{dtUposqinS} that
  \[ \frac{\mathd}{\mathd t} \int_{M_t} U_+^{\frac{n}{2}} \mathd \mu_t
     \leqslant - D_1 (n) \int_{M_t} U_+^{\frac{n}{2}} \mathd \mu_t . \]
  This implies $\int_{M_t} U_+^{\frac{n}{2}} \mathd \mu_t \equiv 0$. Therefore,
  $M_t$ satisfies $| \mathring{h} |^2 \leqslant \frac{1}{n^2} | H
  |^2$ for all $t < 0$. Applying the rigidity theorem in the previous section, we obtain the
  conclusion.
\end{proof}

\begin{proof*}{Proof of Theorem \ref{thm-integral}}
  By the assumption, there exists $t_0 < 0$, such that $\int_{M_t} | \mathring{h}
  |^n \mathd \mu_t < C (n)$ for all $t \in (- \infty, t_0)$. Then combining
  the results of the present section, we complete the proof of Theorem 4.
\end{proof*}

\end{document}